\documentclass[12pt]{article}
\usepackage[T1]{fontenc}
\usepackage{lmodern,amsmath,amsthm,amsfonts,amssymb,graphicx,float,microtype,thmtools,tabularx,underscore,mathtools,anyfontsize,thm-restate,graphicx}
\usepackage[usenames,dvipsnames,svgnames,table]{xcolor}
\usepackage[labelsep = period,justification = centering]{caption}
\usepackage{subcaption}
\usepackage{pgf,tikz}
\usepackage{tkz-euclide} 
\tkzSetUpPoint[fill=black, size = 3pt]
\tkzSetUpLine[line width = 0.6pt]
\usepackage[unicode=true]{hyperref}
\hypersetup{
colorlinks,
breaklinks=true,
linkcolor={black},
citecolor={black},
urlcolor={blue!60!black},
pdftitle={Treewidth, Circle Graphs and Circular Drawings},
pdfauthor={Distel--Hickingbotham--Illingworth--Mohar--Wood}
}
\usepackage[noabbrev,capitalise,nameinlink]{cleveref}
\crefname{lem}{Lemma}{Lemmas}
\crefname{thm}{Theorem}{Theorems}
\crefname{cor}{Corollary}{Corollaries}
\crefname{prop}{Proposition}{Propositions}
\crefname{conj}{Conjecture}{Conjectures}
\crefname{open}{Open Problem}{Open Problems}
\crefformat{equation}{(#2#1#3)}
\Crefformat{equation}{Equation #2(#1)#3}
\usepackage[shortlabels]{enumitem}
\setlist[itemize]{topsep=0ex,itemsep=0ex,parsep=0.25ex}
\setlist[enumerate]{topsep=0ex,itemsep=0ex,parsep=0.25ex}
\newcommand{\defn}[1]{\textcolor{Maroon}{\emph{#1}}}

\newcommand{\WW}{\mathcal{W}}

\newcommand{\GG}{\mathcal{G}}

\newcommand{\OO}{\mathcal{O}}

\newcommand{\NN}{\mathbb{N}}
\newcommand{\RR}{\mathbb{R}}
\newcommand{\ZZ}{\mathbb{Z}}
\usepackage[longnamesfirst,numbers,sort&compress]{natbib}
\makeatletter
\def\NAT@spacechar{~}
\makeatother
\usepackage[margin=30mm]{geometry}
\renewcommand{\baselinestretch}{1.1}
\allowdisplaybreaks

\DeclarePairedDelimiter{\abs}{\lvert}{\rvert}
\DeclarePairedDelimiter{\floor}{\lfloor}{\rfloor}

\DeclarePairedDelimiter{\set}{\{}{\}}

\renewcommand{\le}{\leqslant}
\renewcommand{\geq}{\geqslant}
\renewcommand{\leq}{\leqslant}
\renewcommand{\emptyset}{\varnothing}
\renewcommand{\epsilon}{\varepsilon}
\DeclareMathOperator{\dist}{dist}

\DeclareMathOperator{\tw}{tw}

\DeclareMathOperator{\pw}{pw}

\DeclareMathOperator{\CR}{cr}

\DeclareMathOperator{\rad}{rad}
\newcommand{\hajos}{h_{\textnormal{top}}}
\renewcommand{\hajos}{h'}
\renewcommand{\thefootnote}{\fnsymbol{footnote}}
\allowdisplaybreaks
\theoremstyle{plain}
\newtheorem{thm}{Theorem}
\newtheorem{lem}[thm]{Lemma}
\newtheorem*{claim}{Claim}
\newtheorem{cor}[thm]{Corollary}
\newtheorem{prop}[thm]{Proposition}

\begin{document}

\title{\bf\Large Treewidth, Circle Graphs and Circular Drawings}
\author{%
Robert Hickingbotham\footnotemark[1] \qquad
Freddie Illingworth\footnotemark[2] \\
Bojan Mohar\footnotemark[3] \qquad
David R. Wood\footnotemark[1]}
\date{}
\maketitle
\begin{abstract}
A circle graph is an intersection graph of a set of chords of a circle. We describe the unavoidable induced subgraphs of circle graphs with large treewidth. This includes examples that are far from the `usual suspects'. Our results imply that treewidth and Hadwiger number are linearly tied on the class of circle graphs, and that the unavoidable induced subgraphs of a vertex-minor-closed class with large treewidth are the usual suspects if and only if the class has bounded rank-width. Using the same tools, we also study the treewidth of graphs $G$ that have a circular drawing whose crossing graph is well-behaved in some way. In this setting, we show that if the crossing graph is $K_t$-minor-free, then $G$ has treewidth at most $12t-23$ and has no $K_{2,4t}$-topological minor. On the other hand, we show that there are graphs with arbitrarily large Hadwiger number that have circular drawings whose crossing graphs are $2$-degenerate.
\end{abstract}

\footnotetext[0]{\today. MSC classification: 05C83 graph minors, 
05C10 geometric and topological aspects of graph theory,
05C62 geometric and intersection graph representations}

\footnotetext[1]{School of Mathematics, Monash University, Melbourne, Australia  (\textsf{\{\href{mailto:robert.hickingbotham@monash.edu}{robert.hickingbotham},\allowbreak \href{mailto:david.wood@monash.edu}{david.\allowbreak wood}\}@monash.edu}). Research of R.H.\ is supported by an Australian Government Research Training Program Scholarship. Research of D.W\ is supported by the Australian Research Council. }

\footnotetext[2]{Mathematical Institute, University of Oxford, United Kingdom (\textsf{\href{mailto:illingworth@maths.ox.ac.uk}{illingworth@maths.ox.ac.uk}}). Research supported by EPSRC grant EP/V007327/1.}

\footnotetext[3]{Department  of Mathematics, Simon Fraser University, Burnaby, BC, Canada (\textsf{\href{mailto:mohar@sfu.ca}{mohar@sfu.ca}}). Research supported in part by the NSERC Discovery Grant R611450 (Canada).}

 \newpage
\renewcommand{\thefootnote}{\arabic{footnote}}

\section{Introduction}

This paper studies the treewidth of graphs that are defined by circular drawings. Treewidth is the standard measure of how similar a graph is to a tree, and is of fundamental importance in structural and algorithmic graph theory; see \citep{Reed03,HW17,Bodlaender98} for surveys. The motivation for this study is two-fold. See \cref{SectionPrelim} for definitions omitted from this introduction. 

\subsection{Theme \#1: Circle Graphs}

A \defn{circle graph} is the intersection graph of a set of chords of a circle. Circle graphs form a widely studied graph class~\citep{GKMW23,KK97,dF84,DM21,Davies22a,DGS14,Kloks96} and there have been several recent breakthroughs concerning them. In the study of graph colourings, \citet{DM21} showed that circle graphs are quadratically $\chi$-bounded improving upon a previous longstanding exponential upper bound. \citet{Davies22a} further improved this bound to $\chi(G) \in \OO(\omega(G)\log\omega(G))$, which is best possible. Circle graphs are also fundamental to the study of vertex-minors and are conjectured to lie at the heart of a global structure theorem for vertex-minor-closed graph classes (see \citep{McCarty21}). To this end, \Citet*{GKMW23} recently proved an analogous result to the excluded grid minor theorem for vertex-minors using circle graphs. In particular, they showed that a vertex-minor-closed graph class has bounded rankwidth if and only if it excludes a circle graph as a vertex-minor. For further motivation and background on circle graphs, see \citep{Davies22,McCarty21}.

Our first contribution essentially determines when a circle graph has large treewidth. 

\begin{restatable}{thm}{circlelargetw}\label{circlelargetw}
     Let $t\in\NN$ and let $G$ be a circle graph with treewidth at least $12t + 2$. Then $G$ contains an induced subgraph $H$ that consists of\/ $t$ vertex-disjoint cycles $(C_1, \dotsc, C_t)$ such that for all $i<j$ every vertex of $C_i$ has at least two neighbours in $C_j$. Moreover, every vertex of $G$ has at most four neighbours in any $C_i$ $(1\le i\le t)$.
\end{restatable}

Observe that in \cref{circlelargetw} the subgraph $H$ has a $K_t$-minor obtained by contracting each of the cycles $C_i$ to a single vertex, implying that $H$ has treewidth at least $t - 1$. Moreover, since circle graphs are closed under taking induced subgraphs, $H$ is also a circle graph. We now highlight several consequences of \Cref{circlelargetw}.

First, \Cref{circlelargetw} describes the unavoidable induced subgraphs of circle graphs with large treewidth. Recently, there has been significant interest in understanding the induced subgraphs of graphs with large treewidth \citep{LR22,ACV22,ACHS22a,ACHS22b,AACHSV22,ACDHRSV21,ACHS22,PSTT21,ST21,BBDEGHTW,AACHS22}. To date, most of the results in this area have focused on graph classes where the unavoidable induced subgraphs are the following graphs, the \defn{usual suspects}: a complete graph $K_t$, a complete bipartite graph $K_{t, t}$, a subdivision of the ($t \times t$)-wall, or the line graph of a subdivision of the ($t \times t$)-wall (see \cite{ST21} for definitions). Circle graphs do not contain subdivisions of large walls nor the line graphs of subdivisions of large walls and there are circle graphs of large treewidth that do not contain large complete graphs nor large complete bipartite graphs (see \cref{CircleGraphstwBounded}). To the best of our knowledge this is the first result to describe the unavoidable induced subgraphs of the large treewidth graphs in a natural hereditary class when they are not the usual suspects. Later we show that the unavoidable induced subgraphs of graphs with large treewidth in a vertex-minor-closed class $\GG$ are the usual suspects if and only if $\GG$ has bounded rankwidth (see \cref{boundedrankwidth}). 

Second, the subgraph $H$ in \Cref{circlelargetw} is an explicit witness to the large treewidth of $G$ (with only a multiplicative loss). Circle graphs being $\chi$-bounded says that circle graphs with large chromatic number must contain a large clique witnessing this. \Cref{circlelargetw} can therefore be considered to be a treewidth analogue to the $\chi$-boundedness of circle graphs. We also prove an analogous result for circle graphs with large pathwidth (see \cref{CircleGraphLargePW}).

Third, since the subgraph $H$ has a $K_t$-minor, it follows that every circle graph contains a complete minor whose order is at least one twelfth of its treewidth. This is in stark contrast to the general setting where there are $K_5$-minor-free graphs with arbitrarily large treewidth (for example, grids). \Cref{circlelargetw} also implies the following relationship between the treewidth, Hadwiger number and Haj\'{o}s number of circle graphs (see \cref{SectionCircleStructure})\footnote{For a graph class $\GG$, two graph parameters $\alpha$ and $\beta$ are \defn{tied on $\GG$} if there exists a function $f$ such that $\alpha(G)\leq f(\beta(G)) \text{ and } \beta(G)\leq f(\alpha(G))$ for every graph $G\in \GG$. Moreover, $\alpha$ and $\beta$ are \defn{quadratically/linearly tied on $\GG$} if $f$ may be taken to be quadratic/linear.}.

\begin{restatable}{thm}{Tied}\label{Tied}
    For the class of circle graphs, the treewidth and Hadwiger number are linearly tied. Moreover, the Haj\'{o}s number is quadratically tied to both of them. Both `linear' and `quadratic' are best possible.
\end{restatable}

\subsection{Theme \#2: Graph Drawing}

The second thread of this paper aims to understand the relationship between circular drawings of graphs and their crossing graphs. A \defn{circular drawing} (also called \defn{convex} drawing) of a graph places the vertices on a circle with edges drawn as straight line segments. Circular drawings are a well-studied topic; see \citep{KN19,ST99,GK06} for example. The \defn{crossing graph} of a drawing $D$ of a graph $G$ has vertex-set $E(G)$ where two vertices are adjacent if the corresponding edges cross. Circle graphs are precisely the crossing graphs of circular drawings. If a graph has a circular drawing with a well-behaved crossing graph, must the graph itself also be well-behaved? Graphs that have a circular drawing with no crossings are exactly the outerplanar graphs, which have treewidth at most 2. Put another way, outerplanar graphs are those that have a circular drawing whose crossing graph is $K_2$-minor-free. Our next result extends this fact, relaxing `$K_2$-minor-free' to `$K_t$-minor-free'.

\begin{restatable}{thm}{twHadwiger}\label{twHadwiger}
    For every integer $t \geq 3$, if a graph $G$ has a circular drawing where the crossing graph has no $K_t$-minor, then $G$ has treewidth at most $12t - 23$.
\end{restatable}

\cref{twHadwiger} says that $G$ having large treewidth is sufficient to force a complicated crossing graph in every circular drawing of $G$. A topological $K_{2, 4t}$-minor also suffices.

\begin{restatable}{thm}{NoSubdivision}\label{NoK2kSubdivision}
    If a graph $G$ has a circular drawing where the crossing graph has no $K_t$-minor, then $G$ contains no $K_{2,4t}$ as a topological minor.
\end{restatable}

Outerplanar graphs are exactly those graphs that have treewidth at most 2 and exclude a topological $K_{2,3}$-minor. As such, \cref{twHadwiger,NoK2kSubdivision} extends these structural properties of outerplanar graphs to graphs with circular drawings whose crossing graphs are $K_t$-minor-free. We also prove a product structure theorem for such graphs, showing that every graph that has a circular drawing whose crossing graph has no $K_t$-minor is isomorphic to a subgraph of $H \boxtimes K_{\OO(t^3)}$ where $\tw(H)\leq 2$ (see \cref{ProductStructureCircular}).

In the other direction, we consider sufficient conditions for a graph $G$ to have a circular drawing whose crossing graph has no $K_t$-minor. By \cref{twHadwiger,NoK2kSubdivision}, $G$ must have bounded treewidth and no $K_{2,4t}$-topological minor. While these conditions are necessary, we show that they are not sufficient, but that bounded treewidth with bounded maximum degree is; see \cref{TopologicalMinorCounterExample} and \cref{twDegree} in \cref{SectionCircularBoundedtw} for details.

In addition, we show that the assumption in \cref{twHadwiger} that the crossing graph has bounded Hadwiger number cannot be weakened to bounded degeneracy. In particular, we construct  graphs with arbitrarily large complete graph minors that have a circular drawing whose crossing graph is $2$-degenerate (\cref{2Degen}). This result has applications to the study of general (non-circular) graph drawings, and in particular, leads to the solution of an open problem asked by \citet{HW21b}.

Our proofs of \cref{circlelargetw,twHadwiger,Tied} are all based on the same core lemmas proved in \cref{SectionTools}. The results about circle graphs are in \cref{SectionCircleStructure}, while the results about graph drawings are in \cref{SectionCircularStructure}.

\section{Preliminaries}\label{SectionPrelim}

\subsection{Graph Basics}\label{SectionGraphs}

We use standard graph-theoretic definitions and notation;  see \citep{Diestel5}. 

For a tree $T$, a \defn{$T$-decomposition} of a graph $G$ is a collection $\WW = (W_x \colon x \in V(T))$ of subsets of $V(G)$ indexed by the nodes of $T$ such that
(i) for every edge $vw \in E(G)$, there exists a node $x \in V(T)$ with $v,w \in W_x$; and 
(ii) for every vertex $v \in V(G)$, the set $\set{x \in V(T) \colon v \in W_x}$ induces a (connected) subtree of $T$. 
Each set $W_x$ in $\WW$ is called a \defn{bag}. 
The \defn{width} of $\WW$ is $\max\set{\abs{W_x} \colon x \in V(T)} - 1$. 
A \defn{tree-decomposition} is a $T$-decomposition for any tree $T$. 
The \defn{treewidth $\tw(G)$} of a graph $G$ is the minimum width of a tree-decomposition of $G$. 

A \defn{path-decomposition} of a graph $G$ is a $T$-decomposition where $T$ is a path. The \defn{pathwidth $\pw(G)$} of a graph $G$ is the minimum width of a path-decomposition of $G$.

A graph $H$ is a \defn{minor} of a graph $G$ if $H$ is isomorphic to a graph obtained from a subgraph of $G$ by contracting edges. The \defn{Hadwiger number $h(G)$} of a graph $G$ is the maximum integer $t$ such that $K_t$ is a minor of $G$. 

A graph~$\tilde{G}$ is a \defn{subdivision} of a graph~$G$ if~$\tilde{G}$ can be obtained from~$G$ by replacing each edge~${vw}$ by a path~$P_{vw}$ with endpoints~$v$ and~$w$ (internally disjoint from the rest of~$\tilde{G}$). A graph~$H$ is a \defn{topological minor} of~$G$ if a subgraph of~$G$ is isomorphic to a subdivision of~$H$. The \defn{Haj\'{o}s number $\hajos(G)$} of $G$ is the maximum integer $t$ such that $K_t$ is a topological minor of $G$. A graph~$G$ is \defn{$H$-topological minor-free} if~$H$ is not a topological minor of~$G$. 

It is well-known that for every graph $G$,
\begin{equation*}
    \hajos(G) \leq h(G) \leq \tw(G) + 1.
\end{equation*}

A \defn{graph class} is a collection of graphs closed under isomorphism. A \defn{graph parameter} is a real-valued function $\alpha$ defined on all graphs such that $\alpha(G_1)=\alpha(G_2)$ whenever $G_1$ and $G_2$ are isomorphic.

\subsection{Drawings of Graphs}\label{SectionGraphDrawing}

A \defn{drawing} of a graph $G$ is a function $\phi$ that maps each vertex $v\in V(G)$ to a point $\phi(v)\in\RR^2$ and maps each edge $e=vw\in E(G)$ to a non-self-intersecting curve $\phi(e)$ in $\RR^2$ with endpoints $\phi(v)$ and $\phi(w)$, such that:
\begin{itemize}
    \item $\phi(v)\neq \phi(w)$ for all distinct vertices $v$ and $w$;
    \item $\phi(x)\not\in \phi(e)$ for each edge $e = vw$ and each vertex $x\in V(G)\setminus\set{v,w}$;
    \item each pair of edges intersect at a finite number of points: $\phi(e) \cap \phi(f)$ is finite for all distinct edge $e, f$; and
    \item no three edges internally intersect at a common point: for distinct edges $e, f, g$ the only possible element of $\phi(e) \cap \phi(f) \cap \phi(g)$ is $\phi(v)$ where $v$ is a vertex incident to all of $e, f, g$.
\end{itemize}

A \defn{crossing} of distinct edges $e = uv$ and $f = xy$ is a point in $(\phi(e)\cap \phi(f))\setminus\set{\phi(u),\phi(v),\phi(x),\phi(y)}$; that is, an internal intersection point. A \defn{plane graph} is a graph $G$ equipped with a drawing of $G$ with no crossings. 

The \defn{crossing graph} of a drawing $D$ of a graph $G$ is the graph $X_D$ with vertex set $E(G)$, where for each crossing between edges $e$ and $f$ in $D$, there is an edge of $X_D$ between the vertices corresponding to $e$ and $f$. Note that $X_D$ is actually a multigraph, where the multiplicity of $ef$ equals the number of times $e$ and $f$ cross in $D$. In most drawings that we consider, each pair of edges cross at most once, in which case $X_D$ has no parallel edges. 

Numerous papers have studied graphs that have a drawing whose crossing graph is well-behaved in some way. Here we give some examples. The \defn{crossing number $\CR(G)$} of a graph $G$ is the minimum number of crossings in a drawing of $G$; see the surveys~\citep{Schaefer22,PT00,Szekely04} or the monograph~\citep{Schaefer18}. Obviously, $\CR(G)\leq k$ if and only if $G$ has a drawing $D$ with $\abs{E(X_D)} \leq k$.  \Citet{Tutte63a} defined the \defn{thickness} of a graph $G$ to be the minimum number of planar graphs whose union is $G$; see \cite{Hobbs69,MOS98} for surveys. Every planar graph can be drawn with its vertices at prespecified locations~\citep{Halton91,PW01}. It follows that a graph $G$ has thickness at most $k$ if and only if $G$ has a drawing $D$ such that  $\chi(X_D)\leq k$. A graph is \defn{$k$-planar} if $G$ has a drawing $D$ in which every edge is in at most $k$ crossings; that is, $X_D$ has maximum degree at most $k$; see~\citep{PachToth97,GB07,DMW17,DEW17} for example. More generally, \citet{EG17} defined a graph $G$ to be \defn{$k$-degenerate crossing} if $G$ has a drawing $D$ in which $X_D$ is $k$-degenerate. \Citet{GapPlanar18} defined a graph $G$ to be \defn{$k$-gap-planar} if $G$ has a drawing $D$ in which each crossing can be assigned to one of the two involved edges and each edge is assigned at most $k$ of its crossings. This is equivalent to saying that every subgraph of $X_D$ has average degree at most $2k$. It follows that every $k$-degenerate crossing graph is $k$-gap-planar, and every $k$-gap-planar graph is a $2k$-degenerate crossing graph~\citep{HW21c}. 

A drawing is \defn{circular} if the vertices are positioned on a circle and the edges are straight line segments. A theme of this paper is to study circular drawings $D$ in which $X_D$ is well-behaved in some way. Many papers have considered properties of $X_D$ in this setting. The \defn{convex crossing number} of a graph $G$ is the minimum number of crossings in a circular drawing of $G$; see \citep{Schaefer22} for a detailed history of this topic. Obviously, $G$ has convex crossing number at most $k$ if and only if $G$ has a circular drawing $D$ with $\abs{E(X_D)} \leq k$. The \defn{book thickness} (also called \defn{page-number} or \defn{stack-number}) of a graph $G$ can be defined as the minimum, taken over all circular drawings $D$ of $G$, of $\chi(X_D)$. This parameter is widely studied; see~\citep{DujWoo07,DEHMW22,BBKR17,MBKPRU20,Yann89,Yann20} for example.

\section{Tools}\label{SectionTools}

In this section, we introduce two auxiliary graphs that are useful tools for proving our main theorems. 

For a drawing $D$ of a graph $G$, the \defn{planarisation}, $P_D$, of $D$ is the plane graph obtained by replacing each crossing with a dummy vertex of degree~4. Note that $P_D$ depends upon the drawing $D$ (and not just upon $G$). \Cref{fig:graphplanar} shows a drawing and its planarisation.

\begin{figure}[ht]
    \centering
    \begin{subfigure}[t]{0.4\textwidth}
        \centering
        \includegraphics{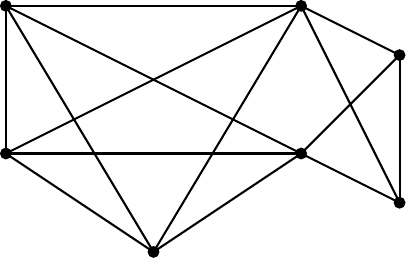}
        \caption{Drawing $D$ of a graph $G$.}
    \end{subfigure}
    \begin{subfigure}[t]{0.4\textwidth}
        \centering
        \includegraphics{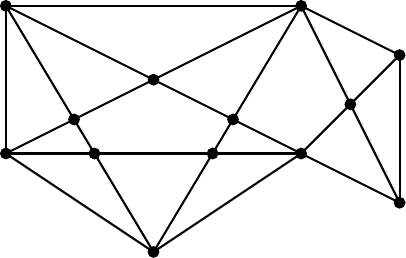}
        \caption{Planarisation $P_D$.}        
    \end{subfigure}
    \caption{A drawing and its planarisation.}\label{fig:graphplanar}
\end{figure}

For a drawing $D$ of a graph $G$, the \defn{map graph}, $M_D$, of $D$ is obtained as follows. First let $P_D$ be the planarisation of $D$. The vertices of $M_D$ are the faces of $P_D$, where two vertices are adjacent in $M_D$ if the corresponding faces share a vertex. If $G$ is itself a plane graph, then it is already drawn in the plane and so we may talk about the map graph, $M_G$, of $G$. Note that all map graphs are connected graphs. \Cref{fig:map} shows the map graph $M_D$ for the drawing $D$ in \cref{fig:graphplanar}.

\begin{figure}[ht]
    \centering
    \includegraphics{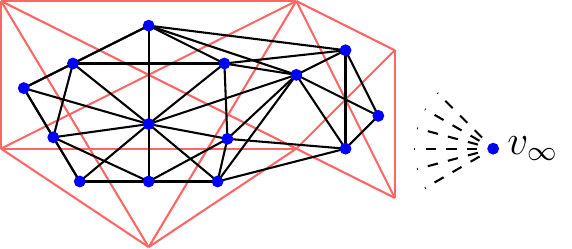}
    \caption{Map graph $M_D$. $v_{\infty}$ is the vertex corresponding to the outer face: it is adjacent to all vertices except the central vertex of degree 10.}\label{fig:map}
\end{figure}

The \defn{radius} of a connected graph $G$, denoted \defn{$\rad(G)$}, is the minimum non-negative integer $r$ such that for some vertex $v \in V(G)$ and for every vertex $w\in V(G)$ we have $\dist_G(v,w)\leq r$.

In \cref{SectionToolsGeneral}, we show that the radius of the map graph $M_D$ acts as an upper bound for the treewidths of $G$ and $X_D$. In \cref{SectionToolsCircular}, we show that if $D$ is a circular drawing and the map graph $M_D$ has large radius, then $X_D$ contains a useful substructure. Thus the radius of $M_D$ provides a useful bridge between the treewidth of $G$, the treewidth of $X_D$, and the subgraphs of $X_D$.%

\subsection{Map Graphs with Small Radii}\label{SectionToolsGeneral}

Here we prove that for any drawing $D$ of a graph $G$, the radius of $M_D$ acts as an upper bound for both the treewidth of $G$ and the treewidth of $X_D$.

\begin{thm}\label{TreewidthMapUB}
    For every drawing $D$ of a graph $G$,
    \begin{align*}
            \tw(G) \leq 6 \rad(M_{D}) + 7 \quad  \textnormal{ and } \quad
            \tw(X_D) \leq 6 \rad(M_D) + 7.
        \end{align*}
\end{thm}

\Citet[Prop.~8.5]{WT07} proved that if a graph $G$ has a circular drawing $D$ such that whenever edges $e$ and $f$ cross, $e$ or $f$ crosses at most $d$ edges, then $G$ has treewidth at most $3d + 11$. This assumption implies $\rad(M_D) \leq \floor{d/2} + 1$ and so the first inequality of \cref{TreewidthMapUB} generalises this result.

It is not surprising that treewidth and radius are related for drawings. A classical result of \citet[(2.7)]{RS-III} says that $\tw(G)\leq 3\rad(G) + 1$ for every connected planar graph $G$. Several authors improved this bound as follows. 

\begin{lem}[\citep{Bod88,DMW17}]\label{PlanarRadius}
    For every connected planar graph $G$,
    \begin{equation*}
        \tw(G) \leq 3\,\rad(G).
    \end{equation*}
\end{lem}

We now prove that if a planar graph $G$ has large treewidth, then the map graph of any plane drawing of $G$ has large radius. For a plane graph $G$, we say a graph $H$ is a \defn{triangulation} of $G$ if $H$ is a plane supergraph of $G$ on the same vertex set and where each face is a triangle.

\begin{lem}\label{PlanarRadiusMap}
    Let $G$ be a plane graph with map graph $M_G$. Then there is a plane triangulation $H$ of $G$ with $\rad(H) \leq \rad(M_G) + 1$. In particular,
    \begin{equation*}
        \tw(G) \leq 3 \,\rad(M_G) + 3.
    \end{equation*}
\end{lem}

\begin{proof}
    Let $F_0$ be a face of $G$ such that every vertex in $M_G$ has distance at most $\rad(M_G)$ from $F_0$. For each face $F$ of $G$, let $\dist_0(F)$ be the distance of $F$ from $F_0$ in $M_G$. 
    
    Fix a vertex $v_0$ of $G$ in the boundary of $F_0$, and set $\rho(v_0) \coloneqq -1$. For every other vertex $v$ of $G$, let
    \begin{equation*}
        \rho(v) = \min\set{\dist_0(F) \colon v \text{ is on the boundary of face } F}.
    \end{equation*}
    Note that $\rho$ takes values in $\set{-1, 0, \dotsc, \rad(M_G)}$. 
    
    We now construct a triangulation $H$ of $G$ such that every vertex $v \neq v_0$ is adjacent (in $H$) to a vertex $u$ with $\rho(u) < \rho(v) $. In particular, the distance from $v$ to $v_0$ in $H$ is at most $\rho(v) + 1 \leq \rad(M_G) + 1$, and so $H$ has the required radius.
    For each face $F$, let $v_F$ be a vertex of $F$ with smallest $\rho$-value. Note that $v_{F_0} = v_0$. Triangulate $G$ as follows. First, consider one-by-one each face $F$. For every vertex $v$ of $F$ that is not already adjacent to $v_F$, add the edge $vv_F$. Finally, let $H$ be obtained by triangulating the resulting graph. 
    
    Consider any vertex $v \neq v_0$. Let $F$ be a face on whose boundary $v$ lies and with $\rho(v) = \dist_0(F)$. If $F = F_0$, then $\rho(v) = 0$ and $vv_0 \in E(H)$, as required. Otherwise assume $F \neq F_0$. By considering a shortest path from $F$ to $F_0$ in $M_G$, there is some face $F' \in V(M_G)$ that shares a vertex with $F$ and has $\dist_0(F') < \dist_0(F)$. Let $v'$ be a vertex on the boundary of both $F$ and $F'$. Then $\rho(v_F) \leq \rho(v') \leq \dist_0(F') < \dist_0(F) = \rho(v)$. Furthermore, by construction, $v$ and $v_F$ are adjacent in $H$, as required.
    
    By \cref{PlanarRadius},
    $\tw(G) \leq \tw(H) \leq 3 \rad(H) \leq 3 \rad(M_G) + 3$.
\end{proof}

Note that a version of \cref{PlanarRadiusMap} with $\rad(M_G)$ replaced by the eccentricity of the outerface in $M_G$ can be proved via outerplanarity\footnote{Say a plane graph $G$ is \defn{$k$-outerplane} if removing all the vertices on the boundary of the outerface leaves a $(k-1)$-outerplane subgraph, where a plane graph is \defn{$0$-outerplane} if it has no vertices. Consider a plane graph $G$, where $v_{\infty}$ is the vertex of $M_G$ corresponding to the outerface. Then one can show that if $v_{\infty}$ has eccentricity $k$ in $M_G$, then $G$ is $(k + 1)$-outerplane, and conversely, if $G$ is $k$-outerplane, then $v_{\infty}$ has eccentricity at most $k$ in $M_G$. \citet{Bodlaender98} showed that every $k$-outerplanar graph has treewidth at most $3k - 1$. The same proof shows that every $k$-outerplane graph has treewidth at most $3k - 1$ (which also follows from \citep{DMW17}).}. 

We use the following lemma about planarisations to extend \cref{PlanarRadiusMap} from plane drawings to arbitrary drawings.

\begin{lem}\label{Planarisationtw}
    For every drawing $D$ of a graph $G$, the planarisation $P_D$ of $D$ satisfies
    \begin{align*}
        \tw(G) \leq 2\tw(P_D) + 1 \quad \textnormal{and} \quad
        \tw(X_D) \leq 2 \tw(P_D) + 1.
    \end{align*}
\end{lem}

\begin{proof}
    Consider a tree-decomposition $(T, \WW)$ of $P_D$ in which each bag has size at most $\tw(P_D) + 1$ . Now we prove the first inequality. Arbitrarily orient the edges of $G$. Each dummy vertex $x$ of $P_D$ corresponds to a crossing between two oriented edges $ab$ and $cd$ of $G$. For each dummy vertex $x$, replace each instance of $x$ in the tree-decomposition $(T, \WW)$ by $b$ and $d$. It is straightforward to verify this gives a tree-decomposition $(T, \WW')$ of $G$ with bags of size at most $2 \tw(P_D) + 2$. Hence $\tw(G) \leq 2\tw(P_D) + 1$.
    
    Now for the second inequality. Each dummy vertex $x$ of $P_D$ corresponds to a crossing between two edges $e$ and $f$ of $G$. For each dummy vertex $x$, replace each instance of $x$ in $(T, \WW)$ by $e$ and $f$. Also, for each vertex $v$ of $G$, delete all instances of $v$ from $(T, \WW)$. This gives a tree-decomposition $(T, \WW'')$ of $X_D$ with bags of size at most $2\tw(P_D) + 2$. Hence $ \tw(X_D) \leq 2 \tw(P_D) + 1$.
\end{proof}

We are now ready to prove \cref{TreewidthMapUB}.

\begin{proof}[Proof of~\cref{TreewidthMapUB}]
    Let $P_D$ be the planarisation of $D$. By definition, $M_D \cong M_{P_D}$. \Cref{PlanarRadiusMap} implies
    \begin{equation*}
        2 \tw(P_D) + 1 \leq 2(3 \rad(M_{P_D}) + 3) + 1 = 6 \rad(M_D) + 7.
    \end{equation*}
    \Cref{Planarisationtw} now gives the required result.
\end{proof}

\subsection{Map Graphs with Large Radii}\label{SectionToolsCircular}

The next lemma is a cornerstone of this paper. It shows that if the map graph of a circular drawing has large radius, then the crossing graph contains a useful substructure. For $a,b\in \RR$ where $a < b$, let \defn{$(a,b)$} denote the open interval $\set{r\in \RR \colon a < r < b}$.

\begin{lem}\label{Cycles}
    Let $D$ be a circular drawing of a graph $G$. If the map graph $M_D$ has radius at least $2t$, then the crossing graph $X_D$ contains $t$ vertex-disjoint induced cycles $C_1, \dotsc, C_t$ such that for all $i<j$ every vertex of $C_i$ has at least two neighbours in $C_j$. 
    Moreover, every vertex of $X_D$ has at most four neighbours in any $C_i$ $(1\le i\le t)$.
\end{lem}

\begin{proof}
    Let $F \in V(M_D)$ be a face with distance at least $2t$ from the outer face of $G$. Let $p$ be a point in the interior of $F$. Let $R_0$ be the infinite ray starting at $p$ and pointing vertically upwards. More generally, for $\theta \in \RR$, let $R_{\theta}$ be the infinite ray with endpoint $p$ that makes a clockwise angle of $\theta$ (radians) with $R_0$. In particular, $R_{\pi}$ is the ray pointing vertically downwards from $p$ and $R_{\theta + 2\pi} = R_{\theta}$ for all $\theta$.
   
    In the statement of the following claim, and throughout the paper, ``cross'' means internally intersect.
    
    \begin{claim}
        Every $R_{\theta}$ crosses at least $2t - 1$ edges of $G$.
    \end{claim}
    
    \begin{proof}
        Consider moving along $R_{\theta}$ from $p$ to the outer face. The distance in $M_D$ only changes when crossing an edge or a vertex of $G$ and changes by at most $1$ when doing so. Since each $R_{\theta}$ contains at most one vertex of $G$, it must cross at least $2t - 1$ edges.
    \end{proof}
    
    For each edge $e$ of $G$, define $I_{e} \coloneqq \set{\theta \colon e \text{ crosses } R_{\theta}}$. Since each edge is a line segment not passing through $p$, each $I_{e}$ is of the form $(a, a') + 2\pi\ZZ$ where $a < a' < a + \pi$. Also note that edges $e$ and $f$ cross exactly if $I_e \cap I_f \neq \emptyset$, $I_e \not \subseteq I_f$, and $I_f \not\subseteq I_e$.

    For a set of edges $E' \subseteq E(G)$ define $I_{E'} = \bigcup \set{I_e \colon e \in E'}$. We say that $E'$ is \defn{dominant} if $I_{E'} = \RR$ and is \defn{minimally dominant} if no proper subset of $E'$ is dominant. Note that if $e, f \in E'$ and $E'$ is minimally dominant, then $e$ and $f$ cross exactly if $I_e \cap I_f \neq \emptyset$.
    
    \begin{claim}
        If $E'$ is minimally dominant, then 
        \begin{enumerate}[label = \textnormal{(\roman*)}]
            \item every $R_{\theta}$ crosses at most two edges of $E'$,
            \item $E'$ induces a cycle in $X_D$,
            \item every edge of $G$ crosses at most four edges of $E'$.
        \end{enumerate}
    \end{claim}

    \begin{proof}
        We first prove (i). Suppose that there is some $R_{\theta}$ crossing distinct edges $e_1, e_2, e_3 \in E'$. Then $\theta \in I_{e_1} \cap I_{e_2} \cap I_{e_3}$ and $\theta + \pi \not \in I_{e_1} \cup I_{e_2} \cup I_{e_3}$. Hence we may write
        \begin{equation*}
            I_{e_i} = (a_i, a'_i) + 2\pi\ZZ, \qquad i = 1, 2, 3
        \end{equation*}
        where $\theta - \pi < a_i < \theta < a'_i < \theta + \pi$. By relabelling, we may assume that $a_1 < a_2 < a_3 < \theta$. Now, if $a'_3$ is not the largest of $a'_1, a'_2, a'_3$, then $(a_3, a'_3) \subseteq (a_1, a'_1) \cup (a_2, a'_2)$ and so $I_{e_3} \subseteq I_{e_1} \cup I_{e_2}$ which contradicts minimality of $E'$. Hence $a'_3 \geqslant a'_1, a'_2$. But then $(a_2, a'_2) \subseteq (a_1, a'_1) \cup (a_3, a'_3)$ and so $I_{e_2} \subseteq I_{e_1} \cup I_{e_3}$ which again contradicts minimality. This proves (i).
    
        We next show that $E'$ induces a connected subgraph of $X_D$. If $E'$ does not, then there is a partition $E_1 \cup E_2$ of $E'$ into non-empty sets such that no edge in $E_1$ crosses any edge in $E_2$. Since $E'$ is minimally dominant, this means $I_{E_1} \cap I_{E_2} = \emptyset$. Consider $\RR$ with the topology induced by the Euclidean metric, which is a connected space. But $I_{E_1}$ and $I_{E_2}$ are non-empty open sets that partition $\RR$. Hence, $E'$ induces a connected subgraph.
        
        We now show that $E'$ induces a 2-regular graph in $X_D$, which together with connectedness establishes (ii). Let $e \in E'$ and write $I_e = (a, a') + 2\pi\ZZ$ where $a < a' < a + \pi$. Since $E'$ is dominant, there are $f, f' \in E'$ with $a \in I_f$ and $a' \in I_{f'}$. If $f = f'$, then $I_e \subseteq I_f$ which contradicts minimality. Hence $f, f'$ are distinct and so $e$ has degree at least two in $X_D$. Suppose that $e$ has some neighbour $f''$ in $X_D$ distinct from $f, f'$. Since $I_{f''}$ is not a subset of $I_e$, it must contain at least one of $a, a'$. By symmetry, we may assume that $I_{f''}$ contains $a$. But then, for some sufficiently small $\epsilon > 0$, all of $I_e, I_f, I_{f''}$ contain $a + \epsilon$ and so $R_{a + \epsilon}$ crosses three edges of $E'$, which contradicts (i). Hence $e$ has exactly two neighbours in $E'$ which establishes (ii).
        
        Finally consider an arbitrary edge $e = uv$ of $G$. Let $R_u$ be the infinite ray from $p$ that contains $u$ and $R_v$ be the infinite ray from $p$ that contains $v$. Observe that every edge of $G$ that crosses $e$ also crosses $R_u$ or $R_v$. By (i), at most four edges in $E'$ cross $e$.
    \end{proof}
    
    For a set of edges $E' \subseteq E(G)$, say an edge $e \in E'$ is \defn{maximal in $E'$} if there is no $f \in E' \setminus \set{e}$ with $I_e \subseteq I_f$. Suppose $E'$ is dominant. Let $E'_{\textnormal{max}}$ be the set of maximal edges in $E'$. Clearly $E'_{\textnormal{max}}$ is still dominant and so has a minimally dominant subset. In particular, every dominant set of edges $E'$ has a subset $E_1$ that is minimally dominant and all of whose edges are maximal in $E'$.
    
    \begin{claim}
        Let $E' \subseteq E(G)$ and $E_1, E_2 \subseteq E'$. Suppose that all the edges of $E_1$ are maximal in $E'$ and that $E_2$ is dominant. Then every edge in $E_1$ crosses at least two edges in $E_2$.
    \end{claim}
    
    \begin{proof}
        Let $e_1 \in E_1$ and write $I_{e_1} = (a, a') + 2\pi\ZZ$ where $a < a' < a + \pi$. Since $E_2$ is dominant, there are $e_2, e_3 \in E_2$ with $a \in I_{e_2}$ and $a' \in I_{e_3}$. If $e_2 = e_3$, then $I_{e_1} \subseteq I_{e_2}$, which contradicts the maximality of $e_1$ in $E'$.
        
        By symmetry, it suffices to check that $e_1$ and $e_2$ cross. Note that for some sufficiently small $\varepsilon > 0$, $a + \varepsilon$ is in both $I_{e_1}$ and $I_{e_2}$ and so $I_{e_1} \cap I_{e_2} \neq \emptyset$. As $a \in I_{e_2} \setminus I_{e_1}$ we have $I_{e_2} \not\subseteq I_{e_1}$. Finally, the maximality of $e_1$ in $E'$ means $I_{e_1} \not\subseteq I_{e_2}$. Hence $e_1$ and $e_2$ do indeed cross.
    \end{proof}
    
    We are now ready to complete the proof. Note that a set of edges is dominant exactly if it crosses every $R_{\theta}$. By the first claim, $E = E(G)$ is dominant. Let $E_1 \subseteq E$ be minimally dominant such that every edge of $E_1$ is maximal in $E$. By part (i) of the second claim, every $R_{\theta}$ crosses at most two edges of $E_1$ and so, by the first claim, crosses at least $2t - 3$ edges of $E \setminus E_1$. Hence, $E \setminus E_1$ is dominant. Let $E_2 \subseteq E \setminus E_1$ be minimally dominant such that every edge of $E_2$ is maximal in $E \setminus E_1$. Continuing in this way we obtain pairwise disjoint $E_1, E_2, \dotsc, E_t \subset E$ such that, for all $i$,
    \begin{itemize}
        \item $E_i$ is minimally dominant,
        \item every edge of $E_i$ is maximal in $E \setminus (\bigcup_{i' < i} E_{i'})$,
        \item every $R_{\theta}$ crosses at most two edges of $E_i$,
        \item every $R_{\theta}$ crosses at least $2(t - i) - 1$ edges of $E \setminus (\bigcup_{i' \leq i} E_{i'})$.
    \end{itemize}
    By part (ii) of the second claim, every $E_i$ induces a cycle $C_i$ in $X_D$. Let $i < j$ and $E' \coloneqq E \setminus (\bigcup_{i' < i} E_{i'})$. Then $E_i, E_j \subseteq E'$ and every edge of $E_i$ is maximal in $E'$. Hence, by the third claim, every edge in $E_i$ crosses at least two edges in $E_j$. In particular, every vertex of $C_i$ has at least two neighbours in $C_j$.
    
    Finally, by part (iii) of the second claim, every vertex of $X_D$ has at most four neighbours in any $C_i$.
\end{proof}

\section{Structural Properties of Circular Drawings}\label{SectionCircularStructure}

\Cref{TreewidthMapUB} says that for any drawing $D$ of a graph $G$, the radius of $M_D$ provides an upper bound for $\tw(G)$ and $\tw(X_D)$. For a general drawing it is impossible to relate $\tw(X_D)$ to $\tw(G)$. Firstly, planar graphs can have arbitrarily large treewidth (for example, the ($n \times n$)-grid has treewidth $n$) and admit drawings with no crossings. In the other direction, $K_{3, n}$ has treewidth 3 and crossing number $\Omega(n^2)$, as shown by \citet{Kleitman70}. In particular, the crossing graph of any drawing of $K_{3, n}$ has average degree linear in $n$ and thus has arbitrarily large complete minors~\citep{Mader67,Mader68} and so arbitrarily large treewidth.

Happily, this is not so for circular drawings. Using the tools in \cref{SectionTools} we show that if a graph $G$ has large treewidth, then the crossing graph of any circular drawing of $G$ has large treewidth. In fact, the crossing graph must contain a large (topological) complete graph minor (see \cref{twHadwiger,twHajos}). In particular, if $X_D$ is $K_t$-minor-free, then $G$ has small treewidth. We further show that if $X_D$ is $K_t$-minor-free, then $G$ does not contain a subdivision of $K_{2, 4t}$ (\cref{NoK2kSubdivision}). Using these results, we deduce a product structure theorem for $G$ (\cref{ProductStructureCircular}).

In the other direction, we ask what properties of a graph $G$ guarantee that it has a circular drawing $D$ where $X_D$ has no $K_t$-minor. Certainly $G$ must have small treewidth. Adding the constraint that $G$ does not contain a subdivision of $K_{2, f(t)}$ is not sufficient (see \cref{TopologicalMinorCounterExample}) but a bounded maximum degree constraint is: we show that if $G$ has bounded maximum degree and bounded treewidth, then $G$ has a circular drawing where the crossing graph has bounded treewidth (\cref{twDegree}).

We also show that there are graphs with arbitrarily large complete graph minors that admit circular drawings whose crossing graphs are $2$-degenerate (see \cref{2Degen}). 

\subsection{Necessary Conditions for \texorpdfstring{$K_t$}{Kt}-Minor-Free Crossing Graphs}

This subsection studies the structure of graphs that have circular drawings whose crossing graph is (topological) $K_t$-minor-free. Much of our understanding of the structure of these graphs is summarised by the next four results (\cref{twHadwiger,NoK2kSubdivision,twHajos,ProductStructureCircular}). 

\twHadwiger*

\NoSubdivision*

\begin{thm}\label{twHajos}
    If a graph $G$ has a circular drawing where the crossing graph has no topological $K_t$-minor, then $G$ has treewidth at most $6t^2 + 6t + 1$.
\end{thm}

From these we may deduce a product structure theorem for graphs that have a circular drawing whose crossing graph is $K_t$-minor-free. For two graphs $G$ and $H$, the \defn{strong product} $G \boxtimes H$ is the graph with vertex-set $V(G) \times V(H)$, and with an edge between two vertices $(v,w)$ and $(v',w')$ if and only if
$v = v'$ and $ww' \in E(H)$, or $w = w'$ and $vv' \in E(G)$, or $vv' \in E(G)$ and $ww' \in E(H)$.
Campbell et al.~\citep[Prop.~55]{UTW} showed that if a graph $G$ is $K_{2, t}$-topological minor-free and has treewidth at most $k$, then $G$ is isomorphic to a subgraph of $H \boxtimes K_{\OO(t^2k)}$ where $\tw(H) \leq 2$. Thus \cref{twHadwiger,NoK2kSubdivision} imply the following product structure result.

\begin{cor}\label{ProductStructureCircular}
    If a graph $G$ has a circular drawing where the crossing graph has no $K_t$-minor, then $G$ is isomorphic to a subgraph of $H \boxtimes K_{\OO(t^3)}$ where $\tw(H)\leq 2$.
\end{cor}

En route to proving these results, we use the cycle structure built by \cref{Cycles} to find (topological) complete minors in the crossing graph of circular drawings. We first show that the treewidth and Hadwiger number of $X_D$ as well as the radius of $M_D$ are all linearly tied.

\begin{lem}\label{CrossingGraphParams} 
    For every circular drawing $D$,
        \begin{equation*}
            \tw(X_D) \leq 6 \rad(M_D) + 7 \leq 12 \, h(X_D) - 11 \leq 12 \tw(X_D) + 1.
        \end{equation*}
\end{lem}

\begin{proof}
    The first inequality is exactly \cref{TreewidthMapUB}, while the final one is the well-known fact that $h(G)\leq \tw(G)+1$ for every graph $G$. To prove the middle inequality we need to show that for any circular drawing $D$,
    \begin{equation}\label{eqRadiusHadwiger}
        \rad(M_D) \leq 2\, h(X_D) - 3.
    \end{equation}
    Let $t \coloneqq h(X_D)$ and suppose, for a contradiction, that $\rad(M_D) \geq 2t - 2$. By \cref{Cycles}, $X_D$ contains $t - 1$ vertex disjoint cycles $C_1, \dotsc, C_{t - 1}$ such that for all $i < j$ every vertex of $C_i$ has a neighbour in $C_j$. Contracting $C_1$ to a triangle and each $C_i$ ($i \geq 2$) to a vertex gives a $K_{t + 1}$-minor in $X_D$. This is the required contradiction.
\end{proof}

Clearly the Haj\'{o}s number of a graph is upper bounded by the Hadwiger number. Our next lemma implies that the Haj\'{o}s number of $X_D$ is quadratically tied to the radius of $M_D$ and to the treewidth and Hadwiger number of $X_D$. 

\begin{lem}\label{RadiusHajos}
    For every circular drawing $D$,
    \begin{equation*}
        \rad(M_D) \leq \hajos(X_D)^2 + 3\, \hajos(X_D) + 1.
    \end{equation*}
\end{lem}

\begin{proof}
   Let $t = \hajos(X_D)+1$ and suppose, for a contradiction, that $\rad(M_D) \geq t^2+t$. By \cref{Cycles}, $X_D$ contains $(t^2+t)/2$ vertex disjoint cycles $C_1, \dotsc, C_{(t^2+t)/2}$ such that for all $i < j$ every vertex of $C_i$ has a neighbour in $C_j$. For each $i\in \set{1,\dots,t}$, let $v_i\in V(C_i)$. We assume that $V(K_t) = \set{1,\dots,t}$ and let $\phi \colon E(K_t)\to \set{t + 1,\dots,(t^2+t)/2}$ be a bijection. Then for each $ij\in E(K_t)$, there is a $(v_i,v_j)$-path $P_{ij}$ in $X_D$ whose internal vertices are contained in $V(C_{\phi(ij)})$. Since $\phi$ is a bijection, it follows that $(P_{ij}\colon ij\in E(K_t))$ defines a topological $K_t$-minor in $X_D$, a contradiction.
\end{proof}

We are now ready to prove \cref{twHadwiger,twHajos}. 

\begin{proof}[Proof of \cref{twHadwiger}]
    Let $D$ be a circular drawing of $G$ with $h(X_D) \leq t - 1$. By \eqref{eqRadiusHadwiger}, $\rad(M_D) \leq 2t - 5$. Finally, by \cref{TreewidthMapUB}, $\tw(G) \leq 12t - 23$.
\end{proof}

\begin{proof}[Proof of \cref{twHajos}]
    Let $D$ be a circular drawing of $G$ with $\hajos(X_D) \leq t - 1$. By \cref{RadiusHajos}, $\rad(M_D) \leq t^2 + t - 1$. Finally, by \cref{TreewidthMapUB}, $\tw(G) \leq 6t^2 + 6t + 1$.
\end{proof}

We now show that the bound on $\tw(G)$ in \cref{twHadwiger} is within a constant factor of being optimal. Let $G_n$ be the ($n \times n$)-grid, which has treewidth $n$ (see \citep{HW17}). \cref{twHadwiger} says that in every circular drawing $D$ of $G_n$, the crossing graph $X_D$ has a $K_t$-minor, where $t = \Omega(n)$. On the other hand, let $D$ be the circular drawing of $G_n$ obtained by ordering the vertices $R_1,R_2,\dots,R_n$, where $R_i$ is the set of vertices in the $i$-th row of $G_n$ (ordered arbitrarily). Let $E_i$ be the set of edges in $G_n$ incident to vertices in $R_i$; note that $\abs{E_i} \leq 3n-1$. If two edges cross, then they have end-vertices in some $E_i$. Thus $(E_1,\dots,E_n)$ is a path-decomposition of $X_D$ of width at most $3n$. In particular, $X_D$ has no $K_{3n+2}$-minor. Hence, the bound on $\tw(G)$ in \cref{twHadwiger} is within a constant factor of optimal. See \cite{SSSV03,SSSV04} for more on circular drawings of grid graphs.

Now we turn to subdivisions and the proof of \cref{NoK2kSubdivision}. As a warm-up, we give a simple proof in the case of no division vertices. 

\begin{prop}\label{NoK2k}
    For every $k\in \NN$, for every circular drawing $D$ of $K_{2, 4k - 1}$, $X_D$ contains $K_{k,k}$ as a subgraph.
\end{prop}

\begin{proof}
    Let the vertex classes of $K_{2, 4k - 1}$ be $X$ and $Y$, where $X = \set{x, y}$ and $\abs{Y} = 4k - 1$. Vertices $x$ and $y$ split the circle into two arcs, one of which must contain at least $2k$ vertices from $Y$. Label these vertices $x, v_1, \dotsc, v_s, y$ where $s \geq 2k$ in order around the circle. For $i \in \set{1,\dots,k}$ define the edges $e_i = yv_i$ and $f_i = xv_{k + i}$. The $e_i$ and $f_i$ are vertices in $X_D$, and for all $i$ and $j$, edges $e_i$ and $f_j$ cross, as required.
\end{proof}

We now work towards the proof of \cref{NoK2kSubdivision}.

A \defn{linear drawing} of a graph $G$ places the vertices on the x-axis with edges drawn as semi-circles above the x-axis. In such a drawing, we consider the vertices of $G$ to be elements of $\RR$ given by their x-coordinate. Such a drawing can be wrapped to give a circular drawing of $G$ with an isomorphic crossing graph.  For an edge $uv\in E(G)$ where $u<v$, define \defn{$I_{uv}$} to be the open interval $(u,v)$. For a set of edges $E' \subseteq E(G)$, define {\defn{$I_{E'}$}\,$\coloneqq  \bigcup \set{I_e \colon e \in E'}$}. Two edges $uv,xy\in E(G)$ where $u<v$ and $x<y$ are \defn{nested} if $u<x<y<v$ or $x<u<v<y$.

\begin{lem}\label{K22ab}
    Let $a,b\in \RR$ where $a<b$, and let $D$ be a linear drawing of a graph $G$ where $G$ consists of two internally vertex-disjoint paths $P_1=(v_1,\dots,v_n)$ and $P_2=(u_1,\dots,u_m)$ such that $u_1,v_1\leq a<b\leq u_m,v_n$. Then there exists $E'\subseteq E(G)$ such that $(a,b)\subseteq I_{E'}$ and $E'$ induces a connected graph in $X_D$. Moreover, for $x\in \set{a,b}$, if $x\not \in V(P_1)\cap V(P_2)$, then $x\in I_{E'}$.
\end{lem}

\begin{proof}
We first show the existence of $E'$. Observe that $(a,b)\subseteq I_{E(P_1)}\cup \set{v_1,\dots,v_n}$. If $G$ contains an edge $uv$ where $u\leq a<b\leq v$, then we are done by setting $E'=\set{uv}$. So assume that $G$ has no edge of that form. Then there is a vertex $v\in V(P_1)$ such that $a<v<b$. 
Each such vertex $v$ is not in $V(P_2)$, implying $v\in I_{E(P_2)}$. Therefore $(a,b)\subseteq I_{E(G)}$. Let $E'$ be a minimal set of edges of $E(G)$ such that $(a,b)\subseteq I_{E'}$. By minimality, no two edges in $E'$ are nested. We claim that $X_D[E']$ is connected. If not, then there is a partition $E_1 \cup E_2$ of $E'$ into non-empty sets such that no edge in $E_1$ crosses any edge in $E_2$. Since $E'$ is minimal, this means $I_{E_1} \cap I_{E_2} = \emptyset$. Consider $(a,b)$ with the topology induced by the Euclidean metric, which is a connected space. But $I_{E_1}\cap (a,b)$ and $I_{E_2}\cap (a,b)$ are non-empty open sets that partition $(a,b)$, a contradiction. Hence, $X_D[E']$ is connected. 
    
Finally, let $x\in \set{a,b}$ and suppose that $x\not \in V(P_1)\cap V(P_2)$. Then $G$ has an edge $uv$ such that $u<x<v$. If $x\in I_{E'}$, then we are done. Otherwise, $E'$ contains an edge incident to $x$. Since $a<u<b$ or $a<v<b$, it follows that $uv$ crosses an edge in $E'$. So adding $uv$ to $E'$ maintains the connectivity of $X_D[E']$ and now $x\in I_{E'}$.
\end{proof}

\begin{lem}\label{K23ab}
    Let $G$ be a subdivision of $K_{2,3}$ and let $x,y\in V(G)$ be the vertices with degree $3$. For every circular drawing $D$ of $G$, there exists a component $Y$ in $X_D$ that contains an edge incident to $x$ and an edge incident to $y$.
\end{lem}

\begin{proof}
    Let $P_1,P_2,P_3$ be the internally disjoint $(x,y)$-paths in $G$. Let $\mathcal{U}=(u_1,\dots,u_{m})$ be the sequence of vertices on the clockwise arc from $x$ to $y$ (excluding $x$ and $y$). Let $\mathcal{V}=(v_1,\dots,u_{n})$ be the sequence of vertices on the anti-clockwise arc from $x$ to $y$ (excluding $x$ and $y$). Say an edge $uv\in E(G)$ is \defn{vertical} if $u\in \mathcal{U}$ and $v\in \mathcal{V}$.
    
    Suppose that no edge of $G$ is vertical. By the pigeonhole principle, we may assume that $V(P_1)\cup V(P_2)\subseteq \mathcal{U}\cup \set{x,y}$. The claim then follows by applying \cref{K22ab} along the clockwise arc from $x$ to $y$.
    
    Now assume that $E(G)$ contains at least one vertical edge. Let $e_1,\dots,e_k$ be an ordering of the vertical edges of $G$ such that if $e_i$ is incident to $u_{i'}$ and $e_{i + 1}$ is incident to $u_{j'}$, then $i'\leq j'$. In the case when $u_{i'}=u_{j'}$, then $e_i$ and $e_{i+1}$ are ordered by their endpoints in $\mathcal{V}$. 
    
    \begin{claim}
        For each $i\in \set{1,\dots,k-1}$, there exists $E_i\subseteq E(G)$ such that $E_i\cup \set{e_i,e_{i+1}}$ induces a connected subgraph of $X_D$.
    \end{claim}
    \begin{proof}
        Clearly the claim holds if $e_i$ and $e_{i + 1}$ cross or if there is an edge in $G$ that crosses both $e_i$ and $e_{i+1}$. So assume that $e_i$ and $e_{i + 1}$ do not cross and no edge crosses both $e_i$ and $e_{i+1}$. Assume $e_i=u'v'$ and $e_{i+1}=u''v''$, where $u', u'' \in \mathcal{U}$ and $v', v'' \in \mathcal{V}$. Let $j\in \set{1,2,3}$. If $P_j$ does not contain $e_i$, then $P_j$ contains neither endpoint of $e_i$. Since $e_i$ separates $x$ from $y$ in the drawing, $P_j$ contains $e_i$ or an edge that crosses $e_i$. Likewise, $P_j$ contains $e_{i+1}$ or an edge that crosses $e_{i+1}$. Let $P_j'=(p_1,\dots,p_m)$ be a vertex-minimal subpath of $P_j$ such that $p_1p_2$ is $e_i$ or crosses $e_i$, and $p_{m-1}p_m$ is $e_{i+1}$ or crosses $e_{i+1}$. By minimality, no edge in $E(P_j')\setminus \set{p_1p_2,p_{m-1}p_m}$ crosses $e_i$ or $e_{i+1}$. Therefore, by the ordering of the vertical edges, no edge in $E(P_j')\setminus \set{p_1p_2,p_{m-1}p_m}$ is vertical. As such, either $\set{p_2,\dots,p_{m-1}}\subseteq \mathcal{U}$ or $\set{p_2,\dots,p_{m-1}}\subseteq \mathcal{V}$. By the pigeonhole principle, without loss of generality, $V(P_1')\cup V(P_2')\subseteq \mathcal{U}$. Since $V(P_1')$ and $V(P_2')$ have distinct endpoints, the claim then follows by applying \cref{K22ab} along the clockwise arc between $u'$ and $u''$.
    \end{proof}
    
    It follows from the claim that all the vertical edges are contained in a single component $Y$ of $X_D$. Now consider the three edges in $G$ incident to $x$. By the pigeonhole principle, without loss of generality, two of these edges are of the form $xu_i, xu_j$ where $i<j$. Let $u_{a}$ be the vertex in $\mathcal{U}$ incident to the vertical edge $e_1$. If $a<j$, then $e_1$ crosses $xu_j$. If $a=j$, then by the ordering of the vertical edges, the path $P_i$ that contains the edge $xu_i$ also contains an edge that crosses both $e_1$ and $xu_j$. Otherwise, $j<a$ and applying \cref{K22ab} to the clockwise arc between $u_j$ and $u_{a}$, it follows that $xu_j$ is also in $Y$. By symmetry, there is an edge incident to $y$ that is in $Y$, as required.
\end{proof}

We are now ready to prove \cref{NoK2kSubdivision}.

\begin{proof}[Proof of \cref{NoK2kSubdivision}]
    Let $G$ be a subdivision of $K_{2, 4t}$ and let $D$ be a circular drawing of $G$. We show that $X_D$ contains a $K_t$-minor. Let $x,y$ be the degree $4t$ vertices in $G$. Let $\mathcal{U}=(u_1,\dots,u_{m})$ be the sequence of vertices on the clockwise arc from $x$ to $y$ (excluding $x$ and $y$). Let $\mathcal{V}=(v_1,\dots,u_{n})$ be the sequence of vertices on the anti-clockwise arc from $x$ to $y$ (excluding $x$ and $y$). Say an edge $uv\in E(G)$ is \defn{vertical} if $u\in \mathcal{U}$ and $v\in \mathcal{V}$.
    
    Let $\ell$ be the number of vertical edges in $G$. Let $k\coloneqq \min \set{\ell,t}$ and let $d\coloneqq t-k$. Then $G$ contains $4d$ paths $P_1,\dots,P_{4d}$ that contain no vertical edge. We say that $P_i$ is a \defn{$\mathcal{U}$-path} (respectively, \defn{$\mathcal{V}$-path}) if it contains an edge incident to a vertex in $\mathcal{U}$ ($\mathcal{V}$). By the pigeonhole principle, without loss of generality, $P_1,\dots,P_{2d}$ are $\mathcal{U}$-paths. By pairing the paths and then applying \cref{K22ab} to the clockwise arc from $x$ to $y$, it follows that $X_D$ contains $d$ vertex-disjoint connected subgraphs $Y_1,\dots,Y_{d}$ in $X_D$ where each $Y_i$ contains an edge (in $G$) incident to $x$ and an edge incident to $y$. Consider distinct $i,j\in \set{1,\dots,d}$. Let $xu_{i'}\in V(Y_i)$ and $xu_{j'}\in V(Y_j)$ and assume that $i'<j'$. Since $xu_{j'}$ separates $u_{i'}$ from $y$ in the drawing, and $P_1,\dots, P_{2d}$ are internally disjoint, it follows that there is an edge in $V(Y_i)$ that crosses $xu_{j'}$. So $Y_1,\dots, Y_d$ are pairwise adjacent, which form a $K_d$-minor in $X_D$. 

    Let $\tilde{E}\coloneqq \set{e_1,\dots,e_k}$ be any set of $k$ vertical edges in $G$. Since $t=d+k$, there are $4k$ internally disjoint $(x,y)$-paths distinct from $P_1,\dots,P_{4d}$, at least $3k$ of which avoid $\tilde{E}$. Grouping these paths into $k$ sets each with three paths, it follows from \cref{K23ab} that there exists $k$ vertex-disjoint connected subgraphs $Z_1,\dots,Z_{k}$ in $X_D$ where each $Z_i$ contains an edge (in $G$) incident to $x$ and an edge incident to $y$. Since each $e\in \tilde{E}$ separates $x$ and $y$ in the drawing, it follows that each $V(Y_i)$ and $V(Z_j)$ contains an edge (in $G$) that crosses $e$. Thus, by contracting each $Y_i$ into a vertex and each $Z_j\cup \set{e_j}$ into a vertex and then deleting all other vertices in $X_D$, we obtain the desired $K_{t}$-minor in $X_D$.
\end{proof}

\subsection{Sufficient Conditions for \texorpdfstring{$K_t$}{Kt}-Minor-Free Crossing Graphs}\label{SectionCircularBoundedtw}

It is natural to consider whether the converse of \cref{twHadwiger,NoK2kSubdivision} holds. That is, does there exist a function $f$ such that if a $K_{2,t}$-topological minor-free graph $G$ has treewidth at most $k$, then there is a circular drawing of $G$ whose crossing graph is $K_{f(t,k)}$-minor-free. Our next result shows that this is false in general. A \defn{$t$-rainbow} in a circular drawing of a graph is a non-crossing matching consisting of $t$ edges between two disjoint arcs in the circle.

\begin{lem}\label{TopologicalMinorCounterExample}
    For every $t\in \NN$, there exists a $K_{2,4}$-topological minor-free graph $G$ with $\tw(G)=2$ such that, for every circular drawing $D$ of $G$, the crossing graph $X_D$ contains a $K_t$-minor.
\end{lem}

\begin{proof}
    Let $T$ be any tree with maximum degree $3$ and sufficiently large pathwidth (as a function of $t$). Such a tree exists as the complete binary tree of height $2h$ has pathwidth $h$. Let $G$ be obtained from $T$ by adding a dominant vertex $v$, so $G$ has treewidth 2. Since $G - v$ has maximum degree $3$, it follows that $G$ is $K_{2,4}$-topological minor-free.
    
    Let $D$ be a circular drawing of $G$ and let $D_T$ be the induced circular drawing of $T$. Since $T$ has sufficiently large pathwidth, a result of \citet[Thm.~2]{Pupyrev20a} implies that $X_D$ has large chromatic number or a $4t$-rainbow\footnote{The result of \citet{Pupyrev20a} is in terms of stacks and queues but is equivalent to our statement.}. Since the class of circle graphs is $\chi$-bounded~\cite{Gyarfas-DM85}, it follows that if $X_D$ has large chromatic number, then it contains a large clique and we are done. So we may assume that $D_T$ contains a $4t$-rainbow. By the pigeonhole principle, there is a subset $\set{a_1b_1,\dots,a_{2t}b_{2t}}$ of the rainbow edges such that $a_ib_i$ topologically separates $v$ from $a_j$ and $b_j$ whenever $i<j$. As such, $a_ib_i$ crosses the edges $va_j$ and $vb_j$ in $D$ whenever $i<j$. Therefore $X_D$ contains a $K_{t,2t}$ subgraph with bipartition $(\set{a_1b_1,\dots,a_tb_t},\set{va_{t + 1},vb_{t + 1},\dots,va_{2t},vb_{2t}})$ and this contains a $K_{t}$-minor. 
\end{proof}

\Cref{TopologicalMinorCounterExample} is best possible in the sense that $K_{2,4}$ cannot be replaced by $K_{2,3}$. An easy exercise shows that every biconnected $K_{2,3}$-topological minor-free graph is either outerplanar or $K_4$. It follows (by considering the block-cut tree) that every $K_{2,3}$-minor-free graph has a circular 1-planar drawing, so the crossing graph consists of isolated edges and vertices.

While $K_{2,k}$-topological minor-free and bounded treewidth is not sufficient to imply that a graph has a circular drawing whose crossing graph is $K_t$-minor-free, we now show that bounded degree and bounded treewidth is sufficient.

\begin{prop}\label{twDegree}
    For $k,\Delta\in\NN$, every graph $G$ with treewidth less than $k$ and maximum degree at most $\Delta$ has a circular drawing in which the crossing graph $X_D$ has treewidth at most $(6\Delta + 1)(18k\Delta)^2 - 1$.
\end{prop}

\begin{proof}
    Refining a method from \citep{DO95,Wood09}, \citet{DW} proved that any such $G$ is isomorphic to a subgraph of $T \boxtimes K_{m}$ where $T$ is a tree with maximum degree $\Delta_T \coloneqq 6\Delta$ and $m \coloneqq 18k\Delta$. Since the treewidth of the crossing graph does not increase when deleting edges and vertices from the drawing, it suffices to show that $T \boxtimes K_{m}$ admits a circular drawing in which the crossing graph $X_D$ has treewidth at most $(\Delta_T + 1)m^2 - 1$. Without loss of generality, assume that $V(K_m) = \set{1,\dots,m}$. Take a circular drawing of $T$ such that no two edges cross (this can be done since $T$ is outerplanar). For each vertex $v \in V(T)$, replace $v$ by $((v,1),\dotsc, (v,m))$ to obtain a circular drawing $D$ of $T\boxtimes K_m$. Observe that if two edges $(u,i)(v,j)$ and $(x,a)(y,b)$ cross in $D$, then $\set{u,v} \cap \set{x,y} \neq \emptyset$. For each vertex $v \in V(T)$, let $W_v$ be the set of edges of $T \boxtimes K_{m}$ that are incident to some $(v,i)$. We claim that $(W_v \colon v\in V(T))$ is a tree-decomposition of $X_D$ with the desired width. Clearly each vertex of $X_D$ is in a bag and for each vertex $e\in V(X_D)$, the set $\set{x\in V(T)\colon e\in W_x}$ induces a graph isomorphic to either $K_2$ or $K_1$ in $T$. Moreover, by the above observation, if $e_1 e_2 \in E(X_D)$, then there exists some node $x \in V(T)$ such that $e_1, e_2\in W_x$. Finally, since there are $\binom{m}{2}$ intra-$K_m$ edges and $\Delta_T \cdot m^2$ cross-$K_m$ edges, it follows that $\abs{W_v} \leq (\Delta_T + 1) m^2$ for all $v\in V(T)$, as required.
\end{proof}

We conclude this subsection with the following open problem:
\begin{quote}
    Does there exist a function $f$ such that every $K_{2,k}$-minor-free graph $G$ has a circular drawing $D$ in which the crossing graph $X_D$ is $K_{f(k)}$-minor-free?
\end{quote}


\subsection{Circular Drawings and Degeneracy}\label{SectionDegeneracy}

\Cref{twHadwiger,twHajos} say that if a graph $G$ has a circular drawing $D$ where the crossing graph $X_D$ excludes a fixed (topological) minor, then $G$ has bounded treewidth. Graphs excluding a fixed (topological) minor have bounded average degree and degeneracy~\citep{Mader67,Mader68}. Despite this, we now show that $X_D$ having bounded degeneracy is not sufficient to bound the treewidth of $G$. In fact, it is not even sufficient to bound the Hadwidger number of $G$.

\begin{thm}\label{2Degen}
    For every $t\in\NN$, there is a graph $G_t$ and a circular drawing $D$ of $G_t$ such that:
    \begin{itemize}
        \item $G_t$ contains a $K_t$-minor,
        \item $G_t$ has maximum degree 3, and
        \item $X_D$ is 2-degenerate.
    \end{itemize}
\end{thm}

\begin{proof}
    We draw $G_t$ with vertices placed on the x-axis (x-coordinate between 1 and $t$) and edges drawn on or above the x-axis. This can then be wrapped to give a circular drawing of $G_t$.
    
    For real numbers $a_1 < a_2 < \dotsb < a_n$, we say a path $P$ is drawn as a \defn{monotone path} with vertices $a_1, \dotsc, a_n$ if it is drawn as follows where each vertex has x-coordinate equal to its label:
    \begin{figure}[ht]
        \centering
        \includegraphics{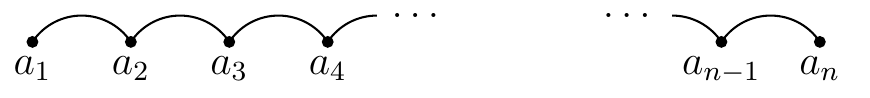}
    \end{figure}
    
    In all our monotone paths, $a_1, a_2, \dotsc, a_n$ will be an arithmetic progression. We construct our drawing of $G_t$ as follows (see \cref{PathByPath} for the construction with $t = 4$). 
    
    \begin{figure}[ht]
    \centering
    \begin{subfigure}{.9\textwidth}
        \centering
        \includegraphics{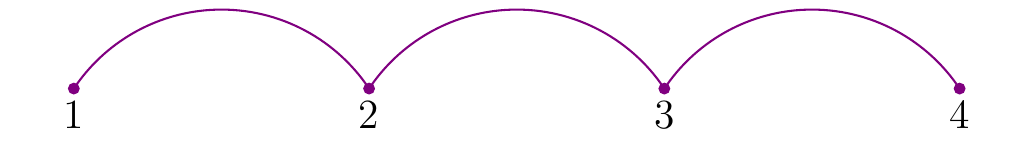}
    \end{subfigure}
    
    \bigskip
    
    \begin{subfigure}{.9\textwidth}
        \centering
        \includegraphics{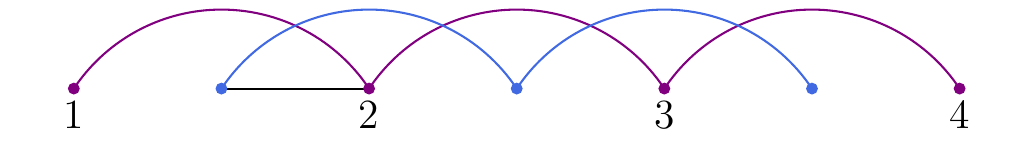}
    \end{subfigure}
    
    \bigskip
    
    \begin{subfigure}{.9\textwidth}
        \centering
        \includegraphics{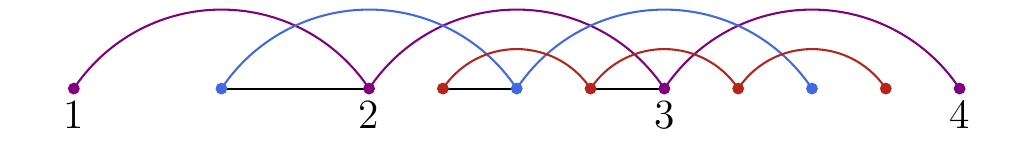}
    \end{subfigure}
    
    \bigskip
    
    \begin{subfigure}{.9\textwidth}
        \centering
        \includegraphics{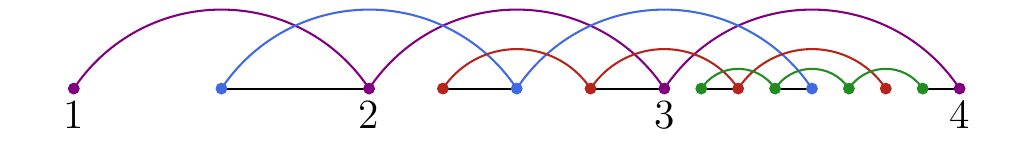}
    \end{subfigure}
    \caption{$G_4$ built up path-by-path, where  $P_0$ is purple, $P_1$ is blue, $P_2$ is red, $P_3$ is green, and the $e_{r, s}$ are black.}
    \label{PathByPath}
\end{figure}

First let $P_0$ be the monotone path with vertices $1, 2, \dotsc, t$. For $s \in \set{1,2,\dots,t-1}$, let $P_s$ be the monotone path with vertices
    \begin{equation*}
        s + 2^{-s}, s + 3 \cdot 2^{-s}, s + 5 \cdot 2^{-s}, \dotsc, t - 2^{-s}.
    \end{equation*}
Observe that these paths are vertex-disjoint. For $0 \leq  r < s \leq t - 1$, let $I_{r,s}$ be the interval 
    \begin{equation*}
        [s + 2^{-r} - 2^{-s}, s + 2^{-r}].
    \end{equation*}
    Note that the lower end-point of $I_{r, s}$ is a vertex in $P_s$ and the upper end-point is a vertex in $P_r$. Also note that no vertex of any $P_i$ lies in the interior of $I_{r, s}$. Indeed, for $i > s$, the vertices of $P_i$ have value at least $s + 2^{-r}$ and for $i \leq s$, the denominator of the vertices of $P_i$ precludes them from being in the interior. Hence for all $r < s$ we may draw a horizontal edge $e_{r, s}$ between the end-points of $I_{r, s}$.
    
    Graph $G_t$ and the drawing $D$ are obtained as a union of the $P_s$ together with all the $e_{r, s}$. The paths $P_s$ are vertex-disjoint and edge $e_{r, s}$ joins $P_r$ to $P_s$, so $G_t$ contains a $K_t$-minor. We now show that the $I_{r, s}$ are pairwise disjoint. Note that $I_{r, s} \subset (s, s + 1]$ so two $I$ with different $s$ values are disjoint. Next note that $I_{r, s} \subset (s + 2^{-(r + 1)}, s + 2^{-r}]$ for $r \leq s - 2$ while $I_{s - 1, s} = [s + 2^{-s}, s + 2^{-(s - 1)}]$ and so two $I$ with the same $s$ but different $r$ values are disjoint. In particular, any vertex $v$ is the end-point of at most one $e_{r, s}$ and so has degree at most three. Hence, $G_t$ has maximum degree three.

    Each edge $e_{r, s}$ is horizontal and crosses no other edges so has no neighbours in $X_D$. Next consider an edge $aa'$ of $P_s$. We have $a' = a + 2 \cdot 2^{-s}$. Exactly one vertex in $V(P_0) \cup V(P_1) \cup \dotsb \cup V(P_{s})$ lies between $a$ and $a'$: their midpoint, $m = a + 2^{-s}$. Vertex $m$ has at most two non-horizontal edges incident to it and so, in $X_D$, every $aa' \in E(P_s)$ has at most two neighbours in $E(P_0) \cup E(P_1) \cup \dotsb \cup E(P_{s})$. Thus $X_D$ is 2-degenerate, as required.
\end{proof}

\subsection{Applications to General Drawings}\label{SectionGlobalDrawings}

This section studies the global structure of graphs admitting a general (not necessarily circular) drawing. In particular, consider the following question: if a graph $G$ has a drawing $D$, then what graph-theoretic assumptions about $X_D$ guarantee that $G$ is well-structured? Even 1-planar graphs contain arbitrarily large complete graph minors~\citep{DEW17}, so one cannot expect $G$ to exclude a fixed minor. 

The following definition works well in this setting. \citet{Eppstein-Algo00} defined a graph class $\GG$ to have
the \defn{treewidth-diameter property}, more recently called \defn{bounded local treewidth}, if there is a function $f$ such that for every graph $G \in \GG$, for every vertex $v\in V(G)$ and for every integer $r\geq 0$, the subgraph of $G$ induced by the vertices at distance at most $r$ from $v$ has treewidth at most $f(r)$. If $f$ is linear (polynomial), then $\GG$ has \defn{linear \textnormal{(}polynomial\textnormal{)} local treewidth}.

\Cref{PlanarRadius} shows that planar graphs have linear local treewidth. More generally, \citet{DMW17} showed that $k$-planar graphs have linear local treewidth (in fact, $k$-planar graphs satisfy a stronger product structure theorem~\citep{DMW}). On the other hand, \citet{HW21b} showed that $1$-gap planar graphs do not have polynomial local treewidth. They also asked whether $k$-gap planar graphs have bounded local treewidth. We show that this is false in a stronger sense. 

\Citet{DGK-LICS} defined a graph class $\GG$ to \defn{locally exclude a minor} if for each $r\in\NN$ there is a graph $H_r$ such that for every graph $G\in\GG$ every subgraph of $G$ with radius at most $r$ contains no $H_r$-minor. Observe that if $\GG$ has bounded local treewidth, then $\GG$ locally excludes a minor. 

By \cref{2Degen}, for each $t\in\NN$, there is a graph $G_t$ that  contains a $K_t$-minor and has a circular drawing $D$ such that $X_D$ is 2-degenerate. Let $G'_t$ be the graph obtained from $G_t$ by adding a dominant vertex into the outer-face of $D$. So the graph $G'_t$ is a $2$-degenerate crossing, has radius $1$, and contains a $K_t$-minor. Thus,  graphs that are $2$-degenerate crossing do not locally exclude a minor, implying they do not have bounded local treewidth, thus answering the above question of \citet{HW21b}. Since every graph that is 2-degenerate crossing is 2-gap-planar, we conclude that 2-gap-planar graphs also do not locally exclude a minor (and do not have bounded local treewidth). This result highlights a substantial difference between $k$-planar graphs and $k$-gap-planar graphs (even for $k=2$). We now prove the following stronger result.

\begin{prop}
\label{StarForestPlanar}
For every $t\in\NN$ there is a graph $G$ and a drawing $D$ of $G$ such that:
\begin{itemize}
    \item $G$ has radius 1,
    \item $G$ contains a $K_{t+1}$-minor, and
    \item $X_D$ is a star-forest.
\end{itemize}    
Thus the graph $G$ is a 1-degenerate crossing and 1-gap-planar.
\end{prop}

\begin{proof}
Let $\phi \colon E(K_t) \to \set{1,\dots,\binom{t}{2}}$ be a bijection. As illustrated in \cref{HorizontalVerticalConstruction}, for each $ij\in E(K_t)$, draw vertices at $(\phi(ij),i), (\phi(ij),j)\in \RR^2$ together with a straight vertical edge between them (red edges in \cref{HorizontalVerticalConstruction}). 

For each $i \in \set{1,2,\dots,t}$, draw a straight horizontal edge between each pair of consecutive vertices along the $y=i$ line. Let $G_0$ be the graph obtained. Let $P_i$ be the subgraph of $G_0$ induced by the vertices on the $y=i$ line. Then $P_i$ is a path on $t-1$ vertices (green edges in \cref{HorizontalVerticalConstruction}). 

For each vertex $v$ in $P_1\cup\dots\cup P_t$ add a `vertical' edge from $v$ to a new vertex $v'$ drawn with y-coordinate $t+1$ (brown edges in \cref{HorizontalVerticalConstruction}). 

    \begin{figure}[!ht]
         \centering
         \includegraphics{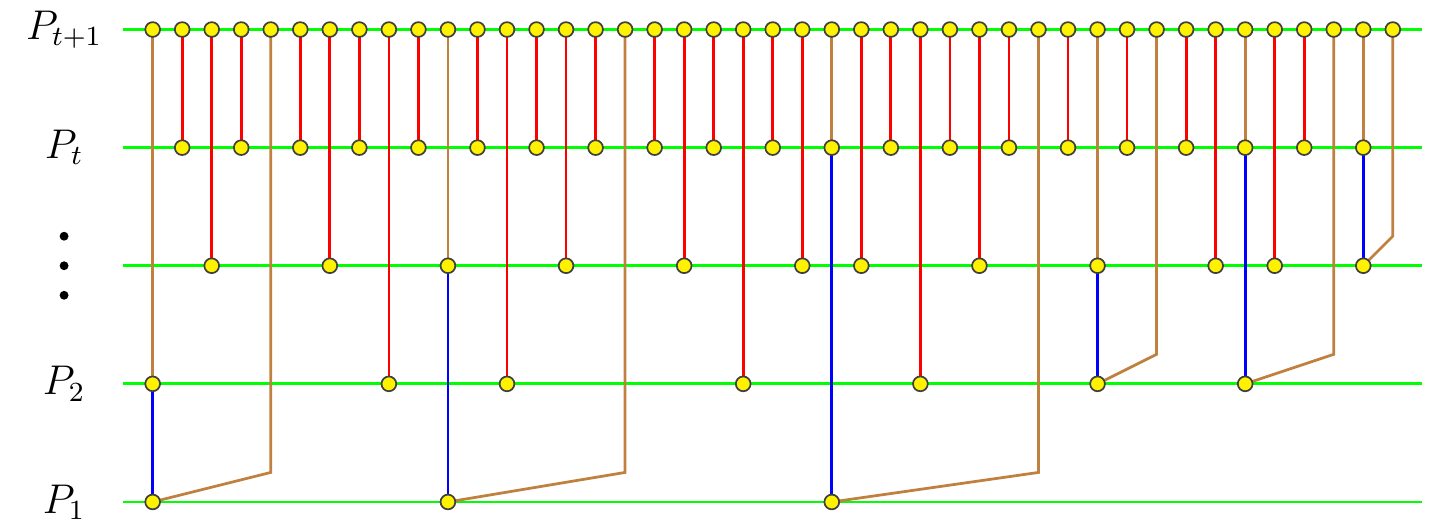}
         \caption{The graph $G_1$ in the proof of \cref{StarForestPlanar}.}
         \label{HorizontalVerticalConstruction} 
    \end{figure}

For $i=1,2,\dots,t$ complete the following step. If two vertical edges $e$ and $f$ cross an edge $g$ in $P_i$ at points $x$ and $y$ respectively, and no other vertical edge crosses $g$ between $x$ and $y$, then subdivide $g$ between $x$ and $y$, introducing a new vertex $v$, and add a new vertical edge from $v$ to a new vertex $v'$ with y-coordinate $t+1$ (red edges in \cref{HorizontalVerticalConstruction}). 

Finally, add a path $P_{t+1}$ through all the vertices with y-coordinate $t+1$. We obtain a graph $G_1$ and a  drawing $D_1$ of $G_1$. Each crossing in $D_1$ is between a vertical and a horizontal edge, and each horizontal edge is crossed by at most one edge. Thus $X_{D_1}$ is a star-forest. 

By construction, no edge in $P_{t+1}$ is crossed in $D_1$, and every vertex has a neighbour in $P_{t+1}$. Thus contracting $P_{t+1}$ to a single vertex gives a graph $G$ with radius 1 and a drawing $D$ of $G$ in which $X_{D}\cong X_{D_1}$. Thus the graph $G$ is 1-degenerate crossing and 1-gap-planar. Finally, $G$ contains a $K_{t+1}$-minor, obtained by contracting each horizontal path $P_i$ into a single vertex.   
\end{proof}

\section{Structural Properties of Circle Graphs}\label{SectionCircleStructure}

Recall that a circle graph is the intersection graph of a set of chords of a circle. More formally, let $C$ be a circle in $\RR^2$. A \defn{chord} of $C$ is a closed line segment with distinct endpoints on $C$. Two chords of $C$ either cross, are disjoint, or have a common endpoint. Let $S$ be a set of chords of a circle $C$ such that no three chords in $S$ cross at a single point. Let $G$ be the crossing graph of $S$. Then $G$ is called a \defn{circle graph}. Note that a graph $G$ is a circle graph if and only if $G \cong X_D$ for some circular drawing $D$ of a graph $H$, and in fact one can take $H$ to be a matching.

We are now ready to prove \cref{circlelargetw,Tied}. While the treewidth of circle graphs has previously been studied from an algorthmic perspective~\cite{Kloks96}, to the best of our knowledge, these theorems are the first structural result on the treewidth of circle graphs.

\circlelargetw*

\begin{proof}
    Let $D$ be a circular drawing of a graph such that $G \cong X_D$. Let $M_D$ be the map graph of $D$. Since $\tw(X_D)=\tw(G)\geq 12t+2$, it follows by \cref{TreewidthMapUB} that $M_D$ has radius at least $2t$. The claim then follows from \cref{Cycles}.
\end{proof}

\Tied*

\begin{proof}
     Let $G$ be a circle graph and let $D$ be a circular drawing with $G \cong X_D$. By \cref{CrossingGraphParams},
    \begin{equation*}
        \tw(G) \leq 6 \rad(M_D) + 7 \leq 12 \, h(G) - 11 \leq 12 \tw(G) + 1.
    \end{equation*}
    So the Hadwiger number and treewidth are linearly tied for circle graphs. This inequality and \cref{RadiusHajos} imply 
    \begin{equation*}
        \hajos(G) - 1 \leq h(G) - 1 \leq \tw(G) \leq 6\rad(M_D) + 7 \leq 6\hajos(G)^2 + 18\, \hajos(G) + 13.
    \end{equation*}
    Hence the Haj\'{o}s number is quadratically tied to both the treewidth and Hadwiger number for circle graphs. Finally, $K_{t, t}$ is a circle graph which has treewidth $t$, Hadwiger number $t + 1$, and Haj\'{o}s number $\Theta(\sqrt{t})$. Hence, `quadratic' is best possible.
\end{proof}

We now discuss several noteworthy consequences of \cref{Tied,circlelargetw}. Recently, there has been significant interest in understanding the unavoidable induced subgraphs of graphs with large treewidth \citep{LR22,ACV22,ACHS22a,ACHS22b,AACHSV22,ACDHRSV21,ACHS22,PSTT21,ST21,AACHS22}. Obvious candidates of unavoidable induced subgraphs include complete graphs, complete bipartite graphs, subdivision of large walls, and line graphs of subdivision of large walls. We say that a hereditary class of graphs $\GG$ is \defn{induced-$\tw$-bounded} if there is a function $f$ such that for every graph $G\in \GG$ with $\tw(G)\geq f(t)$, $G$ contains $K_t$, $K_{t,t}$, a subdivision of the $(t \times t)$-wall, or a line graph of a subdivision of the $(t \times t)$-wall as an induced subgraph\footnote{This definition is motivated by analogy to $\chi$-boundedness; see \citep{SS20}.
Note that while the language of `induced $\tw$-bounded' is original to this paper, \citet{ACHS22b} has previously used this definition under the guise of `special' and \citet{AACHS22} has used it under the guise of `clean'.}. While the class of all graphs is not induced-$\tw$-bounded \cite{AACHSV22,ST21,BBDEGHTW,Pohoata14,Davies22b}, many natural graph classes are. For example, \Citet{AAKST21} showed that every proper minor-closed class is induced-$\tw$-bounded and \citet{Kor22} recently showed that the class of graphs with bounded maximum degree is induced-$\tw$-bounded. We now show that the class of circle graphs is not induced-$\tw$-bounded.

\begin{thm}\label{CircleGraphstwBounded}
    The class of circle graphs is not induced-$\tw$-bounded.
\end{thm}

\begin{proof}
    We first show that for all $t \geq 50$, no circle graph contains a subdivision of the $(t \times t)$-wall or a line graph of a subdivision of the $(t \times t)$-wall as an induced subgraph. As the class of circle graphs is hereditary, it suffices to show that for all $t \geq 50$, these two graphs are not circle graphs. These two graphs are planar (so $K_5$-minor-free) and have treewidth $t \geq 50$. However, \cref{CrossingGraphParams} implies that every $K_5$-minor-free circle graph has treewidth at most 49, which is the required contradiction.
    
    Now consider the family of couples of graphs $((G_t,X_t)\colon t\in \NN)$ given by \cref{2Degen} where $X_t$ is the crossing graph of the drawing of $G_t$. Then $(X_t \colon t\in \NN)$ is a family of circle graphs. Since $(G_t \colon t \in \NN)$ has unbounded treewidth, \cref{twHadwiger} implies that $(X_t \colon t \in \NN)$ also has unbounded treewidth. Moreover, since $X_t$ is $2$-degenerate for all $t\in \NN$, it excludes $K_4$ and $K_{3,3}$ as (induced) subgraphs, as required.
\end{proof}

While the class of circle graphs is not induced-$\tw$-bounded, \cref{circlelargetw} describes the unavoidable induced subgraphs of circle graphs with large treewidth. To the best of our knowledge, this is the first theorem to describe the unavoidable induced subgraphs of a natural hereditary graph class that is not induced-$\tw$-bounded.  In fact, it does so with a linear lower bound on the treewidth of the unavoidable induced subgraphs.

\cref{circlelargetw} can also be used to describe the unavoidable induced subgraphs of circle graphs with large pathwidth.

\begin{thm}\label{CircleGraphLargePW}
    There exists a function $f$ such that every circle graph $G$ with $\pw(G)\geq f(t)$ contains:
    \begin{itemize}
        \item a subdivision of a complete binary tree with height $t$ as an induced subgraph, or
        \item the line graph of a subdivision of a complete binary tree with height $t$ as an induced subgraph, or
        \item an induced subgraph $H$ that consists of\/ $t$ vertex-disjoint cycles $(C_1, \dotsc, C_t)$ such that for all $i<j$ every vertex of $C_i$ has at least two neighbours in $C_j$. Moreover, every vertex of $G$ has at most four neighbours in any $C_i$ $(1\le i\le t)$.
    \end{itemize}
\end{thm}

\begin{proof}
    If $\tw(G)\geq 12t + 2$, then the claim follows from \cref{circlelargetw}. Now assume $\tw(G)<12t + 2$. \citet{Hickingbotham22} showed that there is a function $g(k,t)$ such that every graph with treewidth less than $k$ and pathwidth at least $g(k,t)$ contains a subdivision of a complete binary tree with height $t$ as an induced subgraph or the line graph of a subdivision of a complete binary tree with height $t$ as an induced subgraph. The result follows with $f(t) \coloneqq \max\set{g(12t + 2,t),12t + 2}$.
\end{proof}

We now discuss applications of \cref{circlelargetw} to vertex-minor-closed classes. For a vertex $v$ of a graph $G$, to \defn{locally complement at $v$} means to replace the induced subgraph on the neighbourhood of $v$ by its complement. A graph $H$ is a \defn{vertex-minor} of a graph $G$ if $H$ can be obtained from $G$ by a sequence of vertex deletions and local complementations. Vertex-minors were first studied by \citet{Bouchet87,Bouchet88} under the guise of isotropic systems. The name `vertex-minor' is due to \citet{Oum05}. Circle graphs are a key example of a vertex-minor-closed class.

We now show that a vertex-minor-closed graph class is induced-$\tw$-bounded if and only if it has bounded rank-width. Rank-width is a graph parameter introduced by \citet{OS06} that describes whether a graph can be decomposed into a tree-like structure by simple cuts. For a formal definition and surveys on this parameter, see \citep{Oum17,HOSG08}. \citet{Oum05} showed that rank-width is closed under vertex-minors.

\begin{thm}\label{boundedrankwidth}
    A vertex-minor-closed class $\GG$ is induced-$\tw$-bounded if and only if it has bounded rankwidth.
\end{thm}

\begin{proof}
   Suppose $\GG$ has bounded rankwidth. By a result of \citet*{ACHS22b}, there is a function $f$ such that every graph in $\GG$ with treewidth at least $f(t)$ contains $K_t$ or $K_{t,t}$ as an induced subgraph. Thus $\GG$ is induced-$\tw$-bounded. Now suppose $\GG$ has unbounded rank-width. By a result of \citet*{GKMW23}, $\GG$ contains all circle graphs. It therefore follows by \cref{CircleGraphstwBounded} that $\GG$ is not induced-$\tw$-bounded.
\end{proof}

We conclude with the following question: 
\begin{quote}
    Let $\GG$ be a vertex-minor-closed class with unbounded rank-width. What are the unavoidable induced subgraphs of graphs in $\GG$ with large treewidth? 
\end{quote}
The cycle structure (or variants thereof) in \cref{circlelargetw} must be included in the list of unavoidable induced subgraphs. The case when $\GG$ excludes a large wall and a line graph of a large wall as vertex-minors is of particular interest.

\subsection*{Acknowledgements} This research was initiated at the \href{https://www.matrix-inst.org.au/events/structural-graph-theory-downunder-ll/}{Structural Graph Theory Downunder II} program of the Mathematical Research Institute MATRIX (March 2022). Thanks to all the participants, especially Marc Distel, for helpful conversations.

\fontsize{10pt}{11pt}
\selectfont
\bibliographystyle{DavidNatbibStyle}
\bibliography{main.bbl}
\end{document}